\documentclass[12pt]{amsart}
\usepackage{fullpage}
\usepackage{times}
\usepackage{latexsym}
\usepackage{amssymb}
\usepackage[all]{xy}

\newtheorem{thm}{Theorem}[section]

\newtheorem{cor}[thm]{Corollary}

\newtheorem{lem}[thm]{Lemma}
\newtheorem{que}[thm]{Question}
\newtheorem{prob}[thm]{Problem}

\theoremstyle{remark}
\newtheorem{rem}[thm]{Remark}

\newtheorem{ex}[thm]{Example}

\usepackage{hyperref}

\newcommand{\R}{\mathbb{R}}
\newcommand{\Z}{\mathbb{Z}}

\newcommand{\Out}{\mathrm{Out}}
\newcommand{\Fix}{\mathrm{Fix}}

\title[Anosov diffeomorphisms of products II]
      {Anosov diffeomorphisms of products II. Aspherical manifolds}
\author{Christoforos Neofytidis}
\address{Section de Math\'ematiques, Universit\'e de Gen\`eve, 2-4 rue du Li\`evre, Case postale 64, 1211 Gen\`eve 4, Switzerland}
\email{Christoforos.Neofytidis@unige.ch}
\date{\today}
\subjclass[2010]{37D20, 20F34, 55R10, 57R19, 57M05, 57M10}
\keywords{Anosov diffeomorphism, direct product, aspherical manifold, trivial center, finite outer automorphism group, Hopf property}

\begin{document}

\begin{abstract}
We study aspherical manifolds that do not support Anosov diffeomorphisms. Weakening conditions of Gogolev and Lafont, we show that the product of an infranilmanifold with finitely many aspherical manifolds whose fundamental groups have trivial center and finite outer automorphism group does not support Anosov diffeomorphisms. In the course of our study, we obtain a result of independent group theoretic and topological interest 
on the stability of the Hopf property, namely, that the product of finitely many Hopfian groups with trivial center is Hopfian.
\end{abstract}

\maketitle

\section{Introduction}

In~\cite{NeoAnosov1} we showed that various classes of product manifolds do not support (transitive) Anosov diffeomorphisms, including 
in particular manifolds with non-trivial higher homotopy groups and certain aspherical manifolds. In the present paper, we consider further products of aspherical manifolds. We show that we can remove or relax conditions of Gogolev and Lafont~\cite{GL} on fundamental groups of aspherical manifolds so that their product with an infranilmanifold does not support Anosov diffeomorphisms.

Recall that a diffeomorphism $f$ of a closed oriented smooth $n$-dimensional manifold $M$ is called {\em Anosov} if there exist constants $\mu\in (0,1)$ and $C > 0$, together with a $df$-invariant splitting $TM=E^s\oplus E^u$ of the tangent bundle of $M$, such that for all $m\geq 0$
\begin{equation*}
\begin{split}
\|df^m(v)\| \leq C\mu^m\|v\|,\ v \in E^s,\\  
\|df^{-m}(v)\| \leq C\mu^m\|v\|,\ v \in E^u.
\end{split}
\end{equation*}

The invariant distributions $E^s$ and $E^u$ are called the {\em stable} and {\em unstable} distributions. If either fiber of $E^s$ or $E^u$ has dimension $k$ with $k \leq [n/2]$, then $f$ is called a {\em codimension $k$ Anosov diffeomorphism}. An Anosov diffeomorphism is called {\em transitive} if there exists a point whose orbit is dense in $M$.

Currently, all known examples of Anosov diffeomorphisms are conjugate to affine automorphisms of infranilmanifolds. A long-standing question, going back to Anosov and Smale, asks whether there are any other manifolds that support Anosov diffeomorphisms. Smale suggests, in particular, that if a manifold supports an Anosov diffeomorphism, then it must be covered by a Euclidean space~\cite{Smale}. 
Franks \cite{Fr2} and Newhouse \cite{N} proved that if a manifold admits a codimension one Anosov diffeomorphism, then it is homeomorphic to a torus. Classification results for the existence of Anosov diffeomorphisms on virtually nilpotent manifolds were obtained by Franks, Manning, Brin and others. For instance, Franks~\cite{Fr} and Manning~\cite{Man2} proved that Anosov diffeomorphisms on nilpotent manifolds are conjugate to hyperbolic automorphisms. 
Some major examples of manifolds that do not support Anosov diffeomorphisms include negatively curved manifolds~\cite{Yano,GL}, rational homology spheres~\cite{Sh} and certain manifolds with polycyclic fundamental group
~\cite{Hirsch,Po}. Strong co-homological obstructions to the existence of transitive Anosov diffeomorphisms were given in~\cite{RS}.

A significant question in this study is whether the non-existence of (transitive) Anosov diffeomorphisms on at least one of two given manifolds $M$ and $N$ carries over their direct product $M\times N$. Gogolev and Rodriguez Hertz~\cite{GH} and the author~\cite{NeoAnosov1} provided examples of such products which have non-trivial higher homotopy groups. In~\cite{NeoAnosov1} we proved, in fact, that the product of a negatively curved manifold with a rational homology sphere does not support transitive Anosov diffeomorphisms, which includes as well examples of aspherical manifolds. 
Gogolev and Lafont~\cite{GL} found conditions so that the product of an infranilmanifold with certain aspherical manifolds does not support Anosov diffeomorphisms:

\begin{thm}\label{resultGL}{\normalfont(Gogolev-Lafont \cite{GL}).} 
Let $N$ be a closed infranilmanifold and let $M$ be a closed smooth aspherical manifold whose fundamental group $\Gamma$ 
has the following three properties: 
\begin{itemize}
\item[(i)] $\Gamma$ is Hopfian, 
\item[(ii)] $\Out(\Gamma)$ is finite, 
\item[(iii)] the intersection of all maximal nilpotent subgroups of $\Gamma$ is trivial. 
\end{itemize}
Then $M \times N$ does not support Anosov diffeomorphisms.
\end{thm}

Recall that a group $\Gamma$ is said to be {\em Hopfian} or to have the {\em Hopf property} if every surjective endomorphism of $\Gamma$ is an isomorphism. It is a long-standing problem whether the fundamental group of any aspherical manifold is Hopfian.

As a consequence of the above theorem, Gogolev and Lafont provide some concrete classes of manifolds that do not support Anosov diffeomorphisms, including products of nilpotent manifolds with finitely many negatively curved manifolds:

\begin{cor}{\normalfont(Gogolev-Lafont \cite{GL}).}\label{GLproducts}
Let $N$ be a closed infranilmanifold, and let $M_1,...,M_s$ be a family of closed smooth aspherical manifolds of dimension greater than two, each of which satisfies one of the following properties:
\begin{itemize}
\item[(1)] it has hyperbolic fundamental group, or
\item[(2)] it is an irreducible higher rank locally symmetric space of non-compact type. 
\end{itemize}
Then the product $M_1 \times\cdots\times M_s \times N$ does not support Anosov diffeomorphisms.
\end{cor}

The main goal of this paper is to show that Theorem \ref{resultGL} holds under weaker assumptions.  It is easy to observe that if the intersection of all maximal nilpotent subgroups of a group is trivial, then the center of this group must be trivial as well. Our main result is that Theorem \ref{resultGL} still holds if we relax assumption (iii) of that theorem to the assumption that the center $C(\Gamma)$ of $\Gamma$ is trivial, and without assumption (i) that $\Gamma$ is Hopfian.

\begin{thm}\label{t:main} \
Let $N$ be a closed infranilmanifold and let $M$ be a closed smooth aspherical manifold with fundamental group $\Gamma$ such that 
\begin{itemize}
\item[(a)] $\Out(\Gamma)$ is finite, 
\item[(b)] $C(\Gamma)$ is trivial. 
\end{itemize}
Then $M \times N$ does not support Anosov diffeomorphisms. 
\end{thm}

Gogolev and Lafont prove Corollary \ref{GLproducts} by showing that every finite product of fundamental groups of the two types of families 
of manifolds listed in Corollary \ref{GLproducts} fulfils assumptions (i)-(iii) of Theorem \ref{resultGL}. Those groups have trivial center and thus their product fulfils assumption (b) of Theorem \ref{t:main} as well. In this paper we extend Corollary \ref{GLproducts} to any family of manifolds that fulfils the two conditions of Theorem \ref{t:main}:

\begin{cor}\label{c:extendGL} \
Let $N$ be a closed infranilmanifold and let $M_1,...,M_s$ be closed smooth aspherical manifolds with fundamental groups $\Gamma_i$, $i=1,...,s$, such that 
\begin{itemize}
\item[(a)] $\Out(\Gamma_i)$ is finite, and
\item[(b)] $C(\Gamma_i)$ is trivial. 
\end{itemize}
Then $M_1\times\cdots\times M_s \times N$ does not support Anosov diffeomorphisms.
\end{cor}

Prominent examples of aspherical manifolds whose fundamental groups have trivial center are given by those manifolds with non-zero Euler characteristic~\cite{Got}. Thus Theorem \ref{t:main} and Corollary \ref{c:extendGL} imply the following:

\begin{cor}\label{c:non-zeroEuler}
Let $M_1,...,M_s$ be closed smooth aspherical manifolds with non-zero Euler characteristic such that $\Out(\pi_1(M_i))$ is finite for all $i=1,...,s$. Then $M_1\times\cdots\times M_s \times N$ does not support Anosov diffeomorphisms for any closed infranilmanifold $N$.
\end{cor}

\begin{rem}\label{r:HopfEuler}
In contrast to the examples of Corollary \ref{GLproducts}, it is unknown whether every aspherical manifold with non-zero Euler characteristic has Hopfian fundamental group.
\end{rem}

In order to conclude Corollary \ref{c:extendGL} from Theorem \ref{t:main}, we will show that $\Out(\Gamma_1\times\cdots\times\Gamma_s)$ is finite. In the course of this proof, we obtain the following purely group theoretic result of independent interest on the stability of the Hopf property under taking direct products, which generalizes (and simplifies) the idea of the proof for the Hopf property of products of the groups of Corollary \ref{GLproducts}:

\begin{thm}\label{t:Hopf}
If $\Gamma_1,...,\Gamma_s$ are Hopfian groups with trivial center, then $\Gamma_1\times\cdots\times \Gamma_s$ is Hopfian. 
\end{thm}

This result has also the following interesting consequence with respect to a topological problem of Hopf~\cite[Problem 5.26]{Kirby}; see Section \ref{s:stabilityHopf}:

\begin{cor}\label{c:Hopf}
Let $M_1,...,M_s$ be closed oriented aspherical manifolds whose fundamental groups are Hopfian and have trivial center. Then every self-map of $M_1\times\cdots\times M_s$ of degree $\pm 1$ is a homotopy equivalence.
\end{cor}

\subsection*{Outline of the paper}
In Section \ref{s:weaken}, we explain why we can remove or relax certain conditions of Theorem \ref{resultGL}, obtaining therefore Theorem \ref{t:main}. In Section \ref{s:stabilityHopf}, we discuss the Hopf property for groups with trivial center and prove Corollary \ref{c:extendGL}. Finally, in Section \ref{s:examples}, we discuss examples of aspherical manifolds which do not support Anosov diffeomorphisms.

\subsection*{Acknowledgments}
I am grateful to Andrey Gogolev and to Jean-Fran\c cois Lafont for several fruitful discussions, as well as to Karel Dekimpe for pointing out to me Lemma \ref{l:polycyclic} and to Pierre de la Harpe for useful comments. The support of the Swiss National Science Foundation, under project ``Volumes, growth and characteristic numbers",  is also gratefully acknowledged.

\section{Weakening the assumptions of Theorem \ref{resultGL}}\label{s:weaken}

In this section, we show that Theorem \ref{resultGL} still holds if we remove assumption (i) on the Hopf property of $\Gamma$ and replace assumption (iii) of that theorem with assumption (b) of Theorem \ref{t:main}. We point out that the proof of Theorem \ref{resultGL} in~\cite{GL} consists of several steps which will not be repeated here. Instead, we will simplify some of those group theoretic steps and drop out unnecessary assumptions. 

\subsection{Intersection of all maximal nilpotent subgroups vs trivial center}

First, let us indicate that indeed assumption (iii) of Theorem \ref{resultGL} 
\begin{center}
{\em ``the intersection of all maximal nilpotent subgroups of $\Gamma$ is trivial"}
\end{center}
 is at least as strong as assumption (b) of Theorem \ref{t:main} 
 \begin{center}
 {\em ``the center of $\Gamma$ is trivial"}.
 \end{center}

\begin{lem}\label{l:maxtrivial}
Let $\Gamma$ be a group such that the intersection of all its maximal nilpotent subgroups is trivial. Then $C(\Gamma)=1$.
\end{lem}
\begin{proof}
We can clearly assume that $\Gamma$ is not nilpotent itself. We will show that
\begin{equation}\label{centercontained}
C(\Gamma)\subseteq \mathop{\bigcap}_{\substack{N_i\subset\Gamma }} N_i,
\end{equation}
where $N_i$ are all maximal nilpotent subgroups of $\Gamma$. 

If the center of $\Gamma$ is trivial, then there is nothing to show. Let $1\neq x\in C(\Gamma)$ and suppose that there exists a maximal nilpotent subgroup $N_j$ of $\Gamma$ such that $x\notin N_j$. Consider the semi-direct product $N_j^x:=N_j\rtimes\langle x\rangle$, where $x$, being central in $\Gamma$, acts trivially on $N_j$ by conjugation. In particular, $N_j^x$ is the direct product $N_j\times\langle x \rangle$. Therefore $N_j^x$ is nilpotent being a direct product of nilpotent groups (note that if a power of $x$ lies in $N_j$, then $N_j$ is of finite index in $N_j\times\langle x \rangle$ and $x$ has finite order). Since $N_j$ is maximal nilpotent and $\Gamma$ is not nilpotent, we conclude that $x=1$. This contradiction completes the proof.
\end{proof}

\begin{rem}
For torsion-free virtually polycyclic groups the converse of Lemma \ref{l:maxtrivial} is also true; see Lemma \ref{l:polycyclic}.
\end{rem}

\subsection{The model isomorphism}

Consider now two groups $\Gamma$ and $G$, such that $C(\Gamma)$ is trivial and $G$ is finitely generated nilpotent. Suppose that there is an isomorphism
\[
f_*\colon \Gamma\times G\longrightarrow \Gamma\times G,
\]

We begin with the following lemma, whose proof is straightforward and is left to the reader:

\begin{lem}\label{l:center}
If $\varphi$ is an automorphism of a group $H$, then $\varphi(C(H))=C(H)$.
\end{lem}

We now show that $f_*$ has the form 
\begin{equation}\label{formofautomorphism}
f_*(\gamma, g) = (\alpha(\gamma), \rho(\gamma)L(g)),  \  (\gamma,g)\in\Gamma\times G,
\end{equation}
where $\alpha\colon \Gamma\longrightarrow\Gamma$ and $L\colon G \longrightarrow G$ are automorphisms and $\rho\colon \Gamma\longrightarrow C(G)$ is a homomorphism into the center of $G$. 

By Lemma \ref{l:center}, we have $f_*(C(\Gamma\times G))=C(\Gamma\times G)$. Since moreover $C(\Gamma)=1$, we deduce that $f_*(C(\Gamma\times G))=f_*(1\times C(G))=1\times C(G)$. Therefore, if $g\in C(G)$, then
\[
f_*(1,g)\in 1\times C(G)\subset \Gamma\times G.
\]
Let now $g\in G\setminus C(G)$ and $f_*(1,g)=(\beta,g')\in \Gamma\times G$. Denote the quotient $G/C(G)$ by $G_1$ which is again a finitely generated nilpotent group; see for example~\cite[Proposition 6.19]{NeoIIPP}. Let the induced automorphism
\begin{eqnarray*}
 {f_1}_*\colon \Gamma\times G_1 &\longrightarrow & \Gamma\times G_1\\
     {[(\gamma,h)]}_1 &\mapsto & {[f_*(\gamma,h)]}_1, \ 
\end{eqnarray*}   
where ${[(\gamma,h)]}_1=(\gamma,[h]_1)\in\Gamma\times G_1$. As before, we have that ${f_1}_*(1\times C(G_1))=1\times C(G_1)$. Therefore, if $[g]_1\in C(G_1)$, then ${f_1}_*(1,[g]_1)\in 1\times C(G_1)$ and so $\beta=1$. Thus
\[
f_*(1,g)\in 1\times G\subset \Gamma\times G.
\]
Next assume that $[g]_1\in G_1\setminus C(G_1)$ and consider the nilpotent quotient group $G_2:=G_1/C(G_1)$. We now use the induced automorphism 
\begin{eqnarray*}
 {f_2}_*\colon \Gamma\times G_2 &\longrightarrow & \Gamma\times G_2\\
     {[(\gamma,h)]}_2 &\mapsto & {[{f_1}_*(\gamma,[h]_1)]}_2 
\end{eqnarray*}   
where ${[(\gamma,h)]}_2=(\gamma,[h]_2)\in\Gamma\times G_2$, to show that $f_*(1,g)\in 1\times G$ when $[g]_2\in C(G_2)$, because ${f_2}_*(1\times C(G_2))=1\times C(G_2)$. We continue the process as above and since $G$ is finitely generated nilpotent, there is a $k$ such that $G_k$ is finitely generated Abelian. In particular, 
\[
C(\Gamma\times G_k)=1\times G_k,
\]
showing that $f_*(1,g)\in 1\times G$ for every $[g]_k\in G_k$. We deduce that $f_*(1\times G)=1\times G$, i.e. $f_*$ restricts to an automorphism of $G$. Let us denote this automorphism by $L$.

Next, for each $\gamma\in\Gamma$ we have 
\[
f_*(\gamma,1)=(\alpha(\gamma),\rho(\gamma))
\]
for some homomorphisms $\alpha\colon\Gamma\longrightarrow\Gamma$ and $\rho\colon\Gamma\longrightarrow G$. Since $(\gamma,1)$ and $(1,g)$ commute with each other, we deduce that $\rho(\gamma)$ commutes with $L(g)$ for all $\gamma\in\Gamma$ and $g\in G$, and therefore $\rho\colon\Gamma\longrightarrow C(G)$ as required. 

The form of $f_*$ is given by
\[
f_*(\gamma,g)=f_*((\gamma,1)(1,g))=f_*(\gamma,1)f_*(1,g)=(\alpha(\gamma),\rho(\gamma)L(g)).
\]
The homomorphism $\alpha\colon\Gamma\longrightarrow\Gamma$ is clearly surjective, so it remains to show that it is injective as well. Let $\gamma\in\Gamma$ such that $\alpha(\gamma)=1$. Then $f_*(\gamma,1)=(1,\rho(\gamma))$. But $\rho(\gamma)\in C(G)$, and so Lemma \ref{l:center} implies that
\[
f_*(\gamma,1)=f_*(1,g),
\]
for some $g\in C(G)$. Since $f_*$ is an isomorphism, we deduce that $\gamma=1$, which means that $\alpha$ is injective.

We have now proved the following, which is Lemma 6.2 in \cite{GL}, but without the assumption that $\Gamma$ is Hopfian and weakening the assumption that the intersection of all maximal nilpotent subgroups of $\Gamma$ is trivial to $C(\Gamma)=1$.

\begin{lem}\label{l:weaken}
If $\Gamma$ has trivial center and $G$ is finitely generated nilpotent, then any automorphism $f_*\colon \Gamma\times G \longrightarrow \Gamma\times G$ has the form given in (\ref{formofautomorphism}), where  $\alpha\colon \Gamma\longrightarrow\Gamma$ and $L\colon G \longrightarrow G$ are automorphisms and $\rho\colon \Gamma\longrightarrow C(G)$ is a homomorphism into the center of $G$.
\end{lem}

Suppose now that there exists an Anosov diffeomorphism
$$f\colon M\times N\longrightarrow M\times N,$$
where $M$ and $N$ are as in Theorem \ref{t:main}. After possibly passing to a finite cover of $N$ and some finite power of $f$, we 
 may assume that $N$ is a nilpotent manifold; see also \cite[Lemma 6.1]{GL}. 
 For simplicity, let us write $\Gamma=\pi_1(M)$ and $G=\pi_1(N)$. 

Since $C(\Gamma)$ is trivial and $G$ is finitely generated nilpotent, Lemma \ref{l:weaken} implies that the induced by $f$ automorphism is given by (\ref{formofautomorphism}), i.e.
\[
f_*(\gamma, g) = (\alpha(\gamma), \rho(\gamma)L(g)),  \  (\gamma,g)\in\Gamma\times G.
\]
Moreover, since $\Out(\Gamma)$ is finite, we can assume, after possibly taking some further power of $f$, that $\alpha$ is an inner automorphism of $\Gamma$.  

Finally, passing to (fiberwise) coverings of $M$ and $N$ and to further iterates of $f$, if necessary, we may assume that the invariant distributions are oriented, $f$ preserves the orientation and $f_*$ still has the form given by (\ref{formofautomorphism}).

\begin{rem}
It is worth pointing out that the model isomorphism given by (\ref{formofautomorphism}) must generally be constructed before passing to orientation finite coverings. In~\cite[page 3011/3012]{GL}, the passing to finite coverings of $M\times N$ and to iterates of $f$, in order to achieve orientability of the invariant distributions and so that $f$ preserves the orientation, is done at the beginning of the proof. But passing to finite coverings before bringing $f_*$ into the model form (\ref{formofautomorphism}) does not seem to guarantee that we will be able to obtain (\ref{formofautomorphism}) at a later stage of the proof, because the group theoretic assumptions on $\Gamma$ might not hold for every finite index subgroup of $\Gamma$. Nevertheless, note that passing to orientation coverings at the beginning of the proof does not affect Corollary \ref{GLproducts}, since any finite index subgroup of the fundamental groups of the manifolds mentioned there fulfils all three assumptions of Theorem \ref{resultGL}; see~\cite[Section 7]{GL}. 
\end{rem}

\subsection{Sketch of the remaining steps of the proof of Theorem \ref{t:main}}

The rest of the proof of Theorem \ref{t:main} is identical to that of Theorem \ref{resultGL} given in~\cite{GL}. We only mention briefly the major steps and refer to~\cite{GL} for the details.

Since $f$ has oriented invariant distributions, the number of fixed points of powers of $f$ can be computed using the Lefschetz number $\Lambda(f)$ of $f$, i.e. the sum of indices of the fixed points of $f$. Namely, for each $m\geq 1$,
\begin{equation}\label{eq.FixLef}
|\Fix(f^m)|=|\Lambda(f^m)|.
\end{equation}
Note that $|\Fix(f^m)|$ can be computed by 
\begin{equation}\label{eq.FixAnosov}
|\Fix(f^m)| = re^{mh_{top}(f)} + o(e^{mh_{top}(f)}),
\end{equation}
where $h_{top}(f)$ is the topological entropy of $f$ and $r$ is the number of transitive basic sets with entropy equal to $h_{top}(f)$. If $f$ is transitive, then $r=1$.  See~\cite{Smale} for more details.

Using a model Anosov diffeomorphism, obtained by the group theoretic reductions given in the preceding subsection, Gogolev and Lafont show that, for each $m\geq 1$,
\begin{equation}\label{eq.LefEul}
\Lambda(f^m)=\chi(M)\Lambda(BL),
\end{equation}
where $BL$ denotes the diffeomorphism of $N$ induced by $L$. The latter equation, together with (\ref{eq.FixLef}) and (\ref{eq.FixAnosov}), already complete the proof of Theorem \ref{t:main} when $\chi(M)=0$. 

If $\chi(M)\neq0$, then (\ref{eq.FixLef}), (\ref{eq.FixAnosov}) and (\ref{eq.LefEul}) give, for all $m\geq1$,
\[
\frac{r}{\chi(M)}e^{mh_{top}(f)} + o(e^{mh_{top}(f)}) = \prod_{\lambda\in \mathrm{Spec}(L)}|1-\lambda^m|,
\]
where the right hand side is due to Manning~\cite{Man2} (the product is taken over all eigenvalues, counted with multiplicity, of the Lie algebra automorphism induced by $L$). This implies that the eigenvalues of $L$ are not roots of unity, and so $L$ is an Anosov automorphism; cf.~\cite{Man2} and \cite[Lemma 6.5]{GL}. Then the last algebraic reduction in~\cite{GL} is that $f_*$ has the form
\[
f_*(\gamma, g) = (\alpha(\gamma), L(g)),  \  (\gamma,g)\in\Gamma\times G.
\]
This, together with Franks' work~\cite{Fr} and further computations on locally maximal hyperbolic sets, allows Gogolev and Lafont to construct a new map
\[
\hat{f}\colon M\times S^k\longrightarrow M\times S^k, \ k:=\dim (N),
\]
which is homotopic to the identity, but has the same set of periodic points as (a lift of) $f$ and the Lefschetz number of ${f}^m$ is unbounded as $m$ goes to infinity. This contradiction finishes the proof of Theorem \ref{t:main}.

\section{Stability of the Hopf property}\label{s:stabilityHopf}

In this section we prove Corollary \ref{c:extendGL}. To this end, we need to show that if $\Gamma_1,...,\Gamma_s$ are groups with trivial center and $\Out(\Gamma_i)$ is finite for every $i=1,...,s$, then $\Gamma_1\times\cdots\times\Gamma_s$ has trivial center and $\Out(\Gamma_1\times\cdots\times\Gamma_s)$ is finite as well. The first property is straightforward, so we focus in showing that $\Out(\Gamma_1\times\cdots\times\Gamma_s)$ is finite. In the course of this study, we obtain results of independent interest on the stability of the Hopf property under taking products and on a problem of Hopf for self-maps of degree $\pm 1$.

\subsection{The Hopf property}

The concept of Hopfian groups has its origins in a purely topological problem of H. Hopf:

\begin{prob}{\normalfont(\cite[Problem 5.26]{Kirby}).}\label{p;Hopf}
Is every self-map of a closed oriented manifold $M$ of degree $\pm 1$ a homotopy equivalence?
\end{prob}

For aspherical manifolds, an even stronger form of Hopf's problem was proposed by the author:

\begin{prob}{\normalfont(\cite[Problem 1.2]{NeoHopfdegrees}).}\label{p;Hopfaspherical}
Is every self-map of non-zero degree of a closed oriented aspherical manifold either a homotopy equivalence or homotopic to a non-trivial covering?
\end{prob}

\begin{rem}
Recall also that the Borel conjecture predicts that every homotopy equivalence between two closed aspherical manifolds is homotopic to a homeomorphism.
\end{rem}

The Hopf property has been studied extensively, and for some classes of groups it has been determined completely whether they are Hopfian or not. For instance, it is clear that every simple group is Hopfian and a classical theorem of Mal'cev says that every finitely generated residually finite group is Hopfian. At the other end, a well-known example of a non-Hopfian group is the Baumslag-Solitar group $B(2,3)$. 

Hopfian groups consist a delicate class of groups, being generally not closed under passing to subgroups or quotient groups. An important question is the study of the closure of the Hopf property under product operations. Dey and H. Neumann~\cite{D-HN} proved that the free product of two finitely generated Hopfian groups is Hopfian. The case of direct products is more subtle. Jones~\cite{Jones} proved that there is a non-Hopfian finitely generated group which is isomorphic to the product of two Hopfian groups. Before that, Corner~\cite{Corner} found examples of Hopfian Abelian groups $G,H$ such that $G\times H$ is not Hopfian and even an example of a Hopfian Abelian group $A$ such that $A\times A$ is not Hopfian. Hirshon investigated extensively the problem of the stability of the Hopf property under taking direct products and obtained several sufficient conditions so that the product of two Hopfian groups is again Hopfian; see for example~\cite{Hir1}.

On the one hand, direct products of finitely generated Abelian groups are Hopfian. On the other hand, the known examples of non-Hopfian direct products suggest that the amount of commutativity in at least one of the factors plays an important role. In particular, some of Hirshon's examples point out the role of the center of one of the factors. In this paper we prove the following result on the stability of the Hopf property:

\begin{thm}{\normalfont(Theorem \ref{t:Hopf}).}
If $\Gamma_1,...,\Gamma_s$ are Hopfian groups with trivial center, then $\Gamma_1\times\cdots\times \Gamma_s$ is Hopfian. 
\end{thm}

\begin{proof}
Without loss of generality we may assume that each of the factors $\Gamma_i$ is not decomposable as a direct product, because if some $\Gamma_i$ is a direct product, then every factor of $\Gamma_i$ is again Hopfian and has trivial center. Let 
\[
\phi\colon \Gamma_1\times\cdots\times\Gamma_s\longrightarrow\Gamma_1\times\cdots\times\Gamma_s
\]
be a surjective homomorphism. For each $i$, consider the projection $\pi_i\colon\Gamma_1\times\cdots\times\Gamma_s\longrightarrow\Gamma_i$, and let $\pi_i\phi(\Gamma_{1}),...,\pi_i\phi(\Gamma_{s})$ be the images of the factors of $\Gamma_1\times\cdots\times\Gamma_s$ in $\Gamma_i$ under $\phi$. 

Since $\phi$ is surjective, there exists a factor $\Gamma_{j_0}$, $j_0\in\{1,...,s\}$, such that $\pi_i\phi(\Gamma_{j_0})$ is not trivial. Also, the subgroups $\pi_i\phi(\Gamma_{1}),...,\pi_i\phi(\Gamma_{s})\subset\Gamma_i$ commute elementwise with each other and their union generates $\Gamma_i$, because $\phi$ is surjective. Now, $\pi_i\phi(\Gamma_{j_0})$ and $\pi_i\phi(\Gamma_{1}\times\cdots \times\Gamma_{j_0-1} \times\Gamma_{j_0+1} \times\cdots\times\Gamma_{s})$ commute elementwise with each other as well, and so the multiplication map
\begin{equation}\label{eq:presentation}
\pi_i\phi(\Gamma_{j})\times\pi_i\phi(\Gamma_{1}\times\cdots\times\Gamma_{j_0-1}\times\Gamma_{j_0+1}\times\cdots\times\Gamma_{s})\longrightarrow\Gamma_i
\end{equation}
is a well-defined epimorphism. The kernel of this homomorphism is (isomorphic to) a central subgroup of $\Gamma_i$, because every element of the kernel must clearly be of the form $(g,g^{-1})\in \pi_i\phi(\Gamma_{j_0})\times\pi_i\phi(\Gamma_{1}\times\cdots \times\Gamma_{j_0-1} \times\Gamma_{j_0+1} \times\cdots\times\Gamma_{s})$ and thus $g$ commutes with all elements in $\Gamma_i$. Since $C(\Gamma_i)$ is trivial, we conclude that $\pi_i\phi(\Gamma_{1}\times\cdots\times \Gamma_{j_0-1}\times \Gamma_{j_0+1}\times\cdots\times\Gamma_{s})$ is trivial, because $\Gamma_i$ is not decomposable as a direct product (and $\pi_i\phi(\Gamma_{j_0})$ is not trivial by assumption). This implies that $\pi_i\phi(\Gamma_{j})$ is trivial for all $j\in\{1,...,j_0-1,j_0+1,...,s\}$. 
Thus, we have that for each $i$ there exists a unique $\Gamma_{ij_0}$ such that 
\[
\phi\vert_{\Gamma_{ij_0}}\colon\Gamma_{ij_0}\longrightarrow\Gamma_i
\] 
is surjective and $j_0\neq j'_0$ whenever $i\neq i'$. In fact, $\phi$ permutes the factors of $\Gamma_1\times\cdots\times\Gamma_s$ and thus defines an element of the symmetric group $\mathrm{Sym}(s)$. This means that some power of $\phi$ is the identity element of $\mathrm{Sym}(s)$, i.e. there exists a $k$ such that $\phi^k(\Gamma_i)=\Gamma_i$ for all $i$ (and $\pi_i\phi^k\vert_{\Gamma_{i'}}$ is trivial for all $i'\neq i$). Since each $\Gamma_i$ is Hopfian, we deduce that 
\[
\phi^k\vert_{\Gamma_i}\colon\Gamma_i\longrightarrow\Gamma_i
\]
is an isomorphism for all $i$ and so $\phi^k$ is an isomorphism. This implies that $\phi$ is an isomorphism, completing the proof.
\end{proof}

The proof of Corollary \ref{c:Hopf} is now straightforward:

\begin{proof}[Proof of Corollary \ref{c:Hopf}]
Let $M_1,...,M_s$ be closed oriented aspherical manifolds whose fundamental groups are Hopfian and have trivial center. Suppose 
\[
f\colon M_1\times\cdots\times M_s\longrightarrow M_1\times\cdots\times M_s
\]
is a map of degree $\pm 1$. Then $\pi_1(f)$ is surjective, and since $\pi_1(M_1\times\cdots\times M_s)$ is Hopfian by Theorem \ref{t:Hopf}, we deduce that $\pi_1(f)$ is an isomorphism. Since $M_1\times\cdots\times M_s$ is aspherical, we conclude that $f$ is a homotopy equivalence.
\end{proof}

\subsection{Proof of Corollary \ref{c:extendGL}}

Suppose now that $\Gamma_1,...,\Gamma_s$ are as in Corollary \ref{c:extendGL}, i.e. $\Out(\Gamma_i)$ is finite and $C(\Gamma_i)=1$ for each $i=1,...,s$. As we have seen in the  proof of Theorem \ref{t:Hopf}, any epimorphism 
\[
\phi\colon\Gamma_1\times\cdots\times\Gamma_s\longrightarrow\Gamma_1\times\cdots\times\Gamma_s
\]
permutes the factors of $\Gamma_1\times\cdots\times\Gamma_s$, defining a homomorphism from the automorphism group of  $\Gamma_1\times\cdots\times\Gamma_s$ to the symmetric group $\mathrm{Sym}(s)$. This homomorphism gives rise to a homomorphism 
\[
\Out(\Gamma_1\times\cdots\times\Gamma_s)\longrightarrow\mathrm{Sym}(s),
\]
whose kernel is $\Out(\Gamma_1)\times\cdots\times\Out(\Gamma_s)$. Since $\Out(\Gamma_i)$ is finite for all $i=1,...,s$, we conclude that $\Out(\Gamma_1\times\cdots\times\Gamma_s)$ is finite. Corollary \ref{c:extendGL} now follows from Theorem \ref{t:main}.

\section{Examples}\label{s:examples}

We end our discussion with a few examples illustrating Corollaries \ref{c:extendGL} and \ref{c:non-zeroEuler}. Most of those examples can be derived using Theorem \ref{resultGL}, however, according to Theorem \ref{t:main}, we do not need anymore to check all assumptions of Theorem \ref{resultGL}. Also, some of our examples point out that it is essential to look at the manifolds themselves, and not to their finite covers. 

\subsection{Non-zero Euler characteristic}

By a classical result of Gottlieb~\cite{Got}, the fundamental group of every aspherical manifold with non-zero Euler characteristic has trivial center. (Note, moreover, that the Euler characteristic is multiplicative under taking products and under passing to finite covers multiplied with the degree of the covering map.) Thus Corollary \ref{c:non-zeroEuler} gives examples of products of finitely many even-dimensional manifolds with any infranilmanifold which do not support Anosov diffeomorphisms. (Recall that odd-dimensional manifolds have vanishing Euler characteristic.) The computation of the Euler characteristic of aspherical manifolds is a long-standing topic in topology. Most notably, the Hopf conjecture asserts that $(-1)^k\chi(M)\geq 0$, where $\dim(M)=2k$. Several computations and estimates of the Euler characteristic of aspherical manifolds have been obtained (see for instance~\cite{Ed} and the references there for examples of aspherical manifolds with non-zero Euler characteristic), but their outer automorphism groups seem to be less understood. 

\begin{ex}
Closed oriented aspherical 4-manifolds with non-zero Euler characteristic and finite outer automorphism group include (real and complex) hyperbolic manifolds (see~\cite{Wall,Ko} and~\cite{B,BF,Paulin} for the non-vanishing of the Euler characteristic and the finiteness of the outer automorphism group respectively) and irreducible manifolds modeled on the $\mathbb{H}^2\times \mathbb{H}^2$ geometry (see~\cite{Wall} and~\cite{Mos} respectively). Thus Corollary \ref{c:non-zeroEuler} implies that the product of finitely many such 4-manifolds with any infranilmanifold does not support Anosov diffeomorphisms.
\end{ex}

\begin{rem}
Note that if the Euler characteristic of an aspherical manifold $M$ is not zero, then the intersection of all maximal nilpotent subgroups of $\pi_1(M)$ is trivial; cf.~\cite[Theorem 1.35 (2)]{Lue}.
\end{rem}

\subsection{Virtually polycyclic manifolds}

We now turn to examples of products $M\times N$, where $M$ has Euler characteristic zero. We especially deal with complete fundamental groups, i.e. groups whose center and outer automorphism group are both trivial.

\begin{ex}[Flat manifolds]\label{ex:flat}
Recall that a Bieberbach group $\Gamma$ is a torsion-free group defined by an extension
\[
1\to\Z^n\to\Gamma\to Q\to 1,
\]
where $Q$ is a finite group (called the holonomy group of $\Gamma$) and $\Z^n$ is a maximal Abelian subgroup of $\Gamma$. The corresponding closed aspherical manifold $M=\R^n/\Gamma$ is called flat manifold and has fundamental group $\Gamma$. A characterization for the existence of Anosov diffeomorphisms of flat manifolds is given in~\cite{Po}. Bieberbach groups with trivial center are discussed in~\cite{Sz2,Sz3} and with finite outer automorphism group in~\cite{Sz1,Sz4}. An example of a flat manifold $M$ with complete fundamental group is given in~\cite{Wa}. By Corollary \ref{c:extendGL}, the product of finitely many copies of $M$ with any infranilmanifold does not support Anosov diffeomorphisms. 
\end{ex}

\begin{rem}
Note that a flat manifold $M$ is virtually a torus $T^n$ which supports Anosov diffeomorphisms and has center $\Z^n$. Thus, for a flat manifold $M$ whose fundamental group has trivial center and finite outer automorphism group, the product $M\times T^n$ does not support Anosov diffeomorphisms (by Theorem \ref{t:main}), but it is finitely covered by the $2n$-torus $T^n\times T^n$, which supports Anosov diffeomorphisms.
\end{rem}

\begin{ex}[Solvable manifolds]
Aspherical manifolds with virtually polycyclic (but not virtually Abelian) complete fundamental groups can be constructed building on~\cite{CR,CRW,RT}. By Corollary \ref{c:extendGL}, products of finitely many such manifolds with any infranilmanifold provide examples of virtually polycyclic manifolds (but not virtual tori) that do not support Anosov diffeomorphisms.

An interesting example of a 7-dimensional closed aspherical solvable manifold $M$ with complete fundamental group is given by Robinson~\cite{Ro} (see also~\cite{Ma}): Let $H$ be a torsion-free nilpotent group defined by 
\[
H = \langle a_1,a_2,a_3,a_4,a_5,a_6 \ | \ [a_1,a_2] =a_6, [a_2,a_3]=a_4, [a_3,a_1]=a_5\rangle.
\]
This group is clearly realised as the fundamental group of a 6-dimensional nilpotent manifold $F$, which is a $T^3$-bundle over $T^3$. Let the automorphism $\theta$ of $H$ given by 
\[
a_1\mapsto a_2, \ a_2\mapsto a_3, \ a_3\mapsto a_1a_2^{-1}, \ a_4\mapsto a_4a_5, \ a_5\mapsto a_6, \ a_6\mapsto a_4.
\]
Then the semi-direct product $H\rtimes_\theta\Z$ is a complete group and is realised as the fundamental group of a closed solvable aspherical manifold $M$ which is an $F$-bundle over the circle; see~\cite[Theorem 8]{Ro} and~\cite[Example 3.3]{Ma}. If $N$ is any closed infranilmanifold, then Corollary \ref{c:extendGL} implies that the product of finitely many copies of $M$ with $N$ is a solvable manifold that does not support Anosov diffeomorphisms.
\end{ex}

Note that virtually polycyclic groups are residually finite~\cite[Chapter 5]{Robook} and thus Hopfian. Moreover, according to the following lemma, which was pointed out to me by K. Dekimpe, together with Lemma \ref{l:maxtrivial}, the properties {\em ``the intersection of all maximal nilpotent subgroups of $\Gamma$ is trivial"} and {\em ``the center of $\Gamma$ is trivial"} are equivalent for torsion-free virtually polycyclic groups. Thus Theorems \ref{resultGL} and \ref{t:main} are equivalent for aspherical manifolds with virtually polycyclic fundamental groups.

\begin{lem}\label{l:polycyclic}
Let $\Gamma$ be a torsion-free virtually polycyclic group. If the center of $\Gamma$ is trivial, then the intersection of all maximal nilpotent subgroups of $\Gamma$ is trivial.
\end{lem}
\begin{proof}
Let $G$ be the intersection of all maximal nilpotent subgroups of $\Gamma$ and suppose that $G\neq 1$. 
Denote by $H$ the unique maximal nilpotent normal subgroup of $\Gamma$ (which is called the Fitting subgroup of $\Gamma$). Since $G$ is a normal subgroup of $\Gamma$, we deduce that $G$ lies in $H$. Moreover, since $G$ is normal in $H$ and not trivial, the intersection
\[
K:=G\cap C(H)
\]
is not trivial. Clearly $K$ is normal in $\Gamma$ and Abelian, that is, $K$ is isomorphic to $\Z^k$ for some $k>0$, because $\Gamma$ is torsion-free.

Let $\mathrm{Aut}(K)$ be the automorphism group of $K$ and define
\begin{eqnarray*}
\varphi\colon\Gamma & \longrightarrow & \mathrm{Aut}(K)\\
\gamma & \mapsto & \varphi(\gamma)(x):=\gamma x\gamma^{-1}, \ \gamma\in\Gamma, x\in K.  \ 
\end{eqnarray*}   
Since $K\cong \Z^k$, we can view each $\varphi(\gamma)$ as a matrix in $\mathrm{GL}(k,\Z)$. For any $\gamma\in\Gamma$, there is a maximal nilpotent subgroup $\Delta$ of $\Gamma$ such that $\gamma\in\Delta$.
Now, $K$ is a normal subgroup of the nilpotent group $\Delta$, and so $\Delta$ acts unipotently by conjugation on $K$. In particular, the matrix $\varphi(\gamma)$ is unipotent. 
Thus, $\varphi(\Gamma)$ is a group of unipotent matrices, and so there is a non-trivial element $x_0\in K$ on which any $\gamma\in\Gamma$ acts trivially, i.e. $\varphi(\gamma)(x_0)=x_0$. This means that $x_0\in C(\Gamma)$, completing the proof.
\end{proof}

Since the examples of~\cite{GL} are groups whose intersection of all maximal nilpotent subgroups is trivial (and thus their center is trivial by Lemma \ref{l:maxtrivial}), we end with the following question:

\begin{que}
Which fundamental groups of aspherical manifolds have trivial center but not trivial intersection of all maximal nilpotent subgroups? 
\end{que}

\bibliographystyle{amsplain}

\begin{thebibliography}{123}

\bibitem{B}
M. Bestvina, {\em Local homology properties of boundaries of groups}, Michigan Math. J. {\bf 43} (1) (1996), 123--139.
 
\bibitem{BF}
M. Bestvina and M. Feighn, {\em Stable actions of groups on real trees}, Invent. Math. {\bf 121} (2), (1995), 287--321.

\bibitem{CR}
P. E. Conner and F. Raymond, {\em Manifolds with few periodic homeomorphisms}, Proceedings of the Second Conference on Compact Transformation Groups,  Part II, pp. 1--75,  Lect. Notes in Math. {\bf 299}, Springer, Berlin 1972.

\bibitem{CRW}
P. E. Conner, F. Raymond and P. J. Weinberger, {\em Manifolds with no periodic maps}, Proceedings of the Second Conference on Compact Transformation Groups,  Part II, pp. 81--108,  Lect. Notes in Math. {\bf 299}, Springer, Berlin 1972.

\bibitem{Corner}
A. L. S. Corner, {\em Three examples on hopficity in torsion-free abelian groups}, Acta Math. {\bf 16} (1965), 303--310.

\bibitem{D-HN}
I. M. S. Dey and H. Neumann, {\em The Hopf property of free products}, Math. Z. {\bf 117} (1970), 325--339.

\bibitem{Ed}
A. L. Edmonds, {\em Aspherical 4-manifolds of odd Euler characteristic}, Preprint: arXiv:1710.06345.

\bibitem{Fr}
J. Franks, {\em Anosov diffeomorphisms on tori}, Trans. Am. Math. Soc. {\bf 145} (1969), 117--124.

\bibitem{Fr2}
J. Franks, {\sl Anosov Diffeomorphisms. Global Analysis}, Proceedings of Symposia in Pure Mathematics, pp. 61--93, AMS, Providence, R.I 1970.

\bibitem{GH}
A. Gogolev and F. Rodriguez Hertz, {\em Manifolds with higher homotopy which do not support Anosov diffeomorphisms}, Bull. Lond. Math Soc. {\bf 46} no. 2, (2014) 349--366.

\bibitem{GL}
A. Gogolev and J.-F. Lafont, {\em Aspherical products which do not support Anosov diffeomorphisms}, Ann. Henri Poincar\'e {\bf 17} (2016), 3005--3026.

\bibitem{Got}
D. H. Gottlieb, {\em A certain subgroup of the fundamental group}, Amer. J. Math. {\bf 87} (1965), 840--856.

\bibitem{Hirsch}
M. W. Hirsch, {\em Anosov maps, polycyclic groups, and homology}, Topology {\bf 10} (1971), 177--183.

\bibitem{Hir1}
R. Hirshon, {\em Some theorems on hopficity}, Trans. Amer. Math. Soc. {\bf 141} (1969), 229--244.

\bibitem{Jones}
J. M. T. Jones, {\em Direct products and the Hopf property}, Collection of articles dedicated to the memory of Hanna Neumann, VI. J. Austral. Math. Soc. {\bf 17} (1974), 174--196.

\bibitem{Kirby}
R. Kirby, {\sl Problems in low-dimensional topology}, Berkeley 1995.

\bibitem{Ko}
D. Kotschick, {\em Remarks on geometric structures on compact complex surfaces}, Topology, {\bf 31} no. 2 (1992), 317--321.

\bibitem{Lue}
W. L\"uck, {\sl $L^2$-invariants: theory and applications to geometry and K-theory}, Springer-Verlag, Berlin 2002.

\bibitem{Ma}
W. Malfait, {\em Model aspherical manifolds with no periodic maps},  Trans. Amer. Math. Soc. {\bf 350} no. 11 (1998), 4693--4708.

\bibitem{Man1}
A. Manning, {\em Anosov diffeomorphisms on nilmanifolds}, Proc. Amer. Math. Soc. {\bf 38} (1973), 423--426.

\bibitem{Man2}
A. Manning, {\em There are no new Anosov diffeomorphisms on tori}, Amer. J. Math. {\bf 96} (1974), 422--429.

\bibitem{Mos}
G. D. Mostow, {\sl Strong Rigidity of Locally Symmetric Spaces}, Annals of math. studies, Princeton University Press 1973.

\bibitem{NeoIIPP}
C.~Neofytidis, {\em Fundamental groups of aspherical manifolds and maps of non-zero degree}, Groups Geom. Dyn. {\bf 12} (2018), 637--677.

\bibitem{NeoAnosov1}
C. Neofytidis, {\em Anosov diffeomorphisms of products I. Negative curvature and rational homology spheres}, Preprint: arXiv:1709.05511.

\bibitem{NeoHopfdegrees}
C. Neofytidis, {\em On a problem of Hopf for circle bundles over aspherical manifolds with hyperbolic fundamental groups},  Preprint: arXiv:1712.03582.

\bibitem{N}
S. E. Newhouse, {\em On codimension one Anosov diffeomorphisms}, Amer. J. Math., {\bf 92} no. 3 (1970), 761--770.

\bibitem{Paulin}
F. Paulin, {\em Outer automorphisms of hyperbolic groups and small actions on R-trees}, Arboreal Group Theory, pp. 331--343. MSRI Publ. {\bf 19}, Springer, New York 1991.

\bibitem{Po}
H. L. Porteous, {\em Anosov diffeomorphisms of flat manifolds}, Topology {\bf 11} (1972), 307--315.

\bibitem{RT}
F. Raymond and J. L. Tollefson, {\em Closed 3-manifolds with no periodic maps} Trans. Amer. Math. Soc. {\bf 221} (1976), 403--418.

\bibitem{Ro}
D. J. S. Robinson, {\em Infinite soluble groups with no outer automorphisms}, Rend. Sem. Mat. Univ. Padova {\bf 62} (1980), 281--294.

\bibitem{Robook}
D. J. S. Robinson, {\sl A course in the theory of groups}, Second edition, Graduate Texts in Mathematics {\bf 80}, Springer-Verlag, New York 1996.

\bibitem{RS}
D. Ruelle and D. Sullivan, {\em Currents, flows and diffeomorphisms}, Topology {\bf 14} (1975), 319--327.

\bibitem{Sh}
K. Shiraiwa, {\em Manifolds which do not admit Anosov diffeomorphisms}, Nagoya Math. J. {\bf 49} (1973), 111--115.

\bibitem{Smale}
S. Smale, {\em Differentiable dynamical systems}, Bull. Amer. Math. Soc. {\bf 73} (1967), 747--817.

\bibitem{Sz1}
A. Szczepa\'nski, {\em Aspherical manifolds with the Q-homology of a sphere}, Mathematika {\bf 30} (1983), 291--294.

\bibitem{Sz2}
A. Szczepa\'nski, {\em Euclidean space form with the first Betti number equal to zero}, Quart. J. Math. Oxford (2) {\bf 36} (1985), 489--493.

\bibitem{Sz3}
A. Szczepa\'nski, {\em Five dimensional Bieberbach groups with trivial center}, Manuscripta Math. {\bf 68} (1990), 191--208.

\bibitem{Sz4}
A. Szczepa\'nski, {\em Outer Automorphism Groups of Bieberbach Groups},  Bull. Belg. Math. Soc. Simon Stevin {\bf 3} no. 5 (1996), 585--593.

\bibitem{Wa}
R. Waldm\"uller, {\em A flat manifold with no symmetries}, Expo. Math. {\bf 12} (2003), 71--77.

\bibitem{Wall}
C. T. C. Wall, {\em Geometric structures on compact complex analytic surfaces}, Topology {\bf 25} no. 2 (1986), 119--153.
 
\bibitem{Yano}
K. Yano, {\em There are no transitive Anosov diffeomorphisms on negatively curved manifolds}, Proc. Japan Acad. Ser. A Math. Sci. {\bf 59} (9) (1983), 445.

\end{thebibliography}

\end{document}